\documentclass[12pt]{article}
\usepackage{epic,latexsym,amssymb}
\usepackage{color}
\usepackage{tikz}
\usepackage[colorinlistoftodos]{todonotes}
\usepackage{pgfplots}

\usepackage{mathrsfs}
\usepackage[colorlinks]{hyperref}
\usepackage{amsfonts,epsf,amsmath}
\usepackage{cite}
\usepackage{lineno}
\usepackage[lined,ruled,linesnumbered]{algorithm2e}
\usepackage{multicol}
\usepackage{subfig}
\usepackage{epic}
\usepackage{amssymb}
\usepackage{latexsym}
\usepackage{tikz}
\usetikzlibrary{decorations.markings}
\usepackage{pgf}

\usetikzlibrary{arrows}

\usepackage{comment}

\newtheorem{theorem}{Theorem}

\newtheorem{lemma}{Lemma}

\newtheorem{corollary}{Corollary}
\newtheorem{observation}{Observation}
\newtheorem{proposition}{Proposition}

\newtheorem{problem}{Problem}
\newcommand{\QED}{$\Box$}

\let\oldenumerate\enumerate
\renewcommand{\enumerate}{
  \oldenumerate
  \setlength{\itemsep}{1pt}
  \setlength{\parskip}{0pt}
  \setlength{\parsep}{0pt}
}

\def\vertex(#1){\put(#1){\circle*{2}}}
\def\vertexo(#1){\put(#1){\circle{2}}}
\def\vert(#1){\put(#1){\circle*{1.5}}}
\def\verto(#1){\put(#1){\circle{1.5}}}
\def\lab(#1)#2{\put(#1){\makebox(0,0)[c]{#2}}}
\begin{document}

\title{The minmin coalition number in graphs}

\author{$^1$Davood Bakhshesh and $^2$Michael A. Henning\thanks{Research supported in part by the University of Johannesburg and the South African National Research Foundation} \\ \\
$^1$Department of Computer Science\\
University of Bojnord \\
Bojnord, Iran \\
\small \tt Email: d.bakhshesh@ub.ac.ir \\
\\
$^2$Department of Mathematics and Applied Mathematics \\
University of Johannesburg \\
Auckland Park, 2006 South Africa\\
\small \tt Email: mahenning@uj.ac.za}

\date{}
\maketitle

\begin{abstract}
A set $S$ of vertices in a graph $G$ is a dominating set if every vertex of $V(G) \setminus S$ is adjacent to a vertex in $S$. A coalition in $G$ consists of two disjoint sets of vertices $X$ and $Y$ of $G$, neither of which is a dominating set but whose union $X \cup Y$ is a dominating set of $G$. Such sets $X$ and $Y$ form a coalition in $G$. A coalition partition, abbreviated $c$-partition, in $G$ is a partition $\mathcal{X} = \{X_1,\ldots,X_k\}$ of the vertex set $V(G)$ of $G$ such that for all $i \in [k]$, each set $X_i \in \mathcal{X}$ satisfies one of the following two conditions: (1) $X_i$ is a dominating set of $G$ with a single vertex, or (2) $X_i$ forms a coalition with some other set $X_j \in \mathcal{X}$.
Let ${\cal A} = \{A_1,\ldots,A_r\}$ and ${\cal B}= \{B_1,\ldots, B_s\}$ be two partitions of $V(G)$. Partition ${\cal B}$ is a refinement of partition ${\cal A}$ if every set $B_i \in {\cal B} $ is either equal to, or a proper subset of, some set $A_j \in {\cal A}$. Further if ${\cal A} \ne {\cal B}$, then ${\cal B}$ is a proper refinement of ${\cal A}$.
%
%
Partition ${\cal A}$ is a minimal $c$-partition if it is not a proper refinement of another $c$-partition. Haynes et al. [AKCE Int. J. Graphs Combin. 17 (2020), no. 2, 653--659] defined the minmin coalition number $c_{\min}(G)$ of $G$ to equal the minimum order of a minimal $c$-partition of $G$. We show that $2 \le c_{\min}(G) \le n$, and we characterize graphs $G$ of order $n$ satisfying $c_{\min}(G) = n$. A polynomial-time algorithm is given to determine if $c_{\min}(G)=2$ for a given graph $G$. A necessary and sufficient condition for a graph $G$ to satisfy $c_{\min}(G) \ge 3$ is given, and a characterization of graphs $G$ with minimum degree~$2$ and $c_{\min}(G)= 4$ is provided.
\end{abstract}


\section{Introduction}

A set $S$ of vertices in a graph $G$ is a \emph{dominating set} if every vertex in $V(G) \setminus S$ is adjacent to a vertex in~$S$. The \emph{domination number} $\gamma(G)$ of $G$ is the minimum cardinality of a dominating set of $G$. If $A,B \subseteq S$, then  set $A$ \emph{dominates} the set $B$ if  every vertex $b \in B$ belongs to $A$ or is adjacent to a vertex of $A$. The study of domination in graphs is an active area of research in graph theory. A thorough treatment of this topic can be found in recent so-called ``domination books''~\cite{HaHeHe-20,HaHeHe-21,HaHeHe-23,HeYe-book}.

For graph theory notation and terminology, we generally follow~\cite{HaHeHe-23}.  Specifically, let $G$ be a graph with vertex set $V(G)$ and edge set $E(G)$, and of order $n(G) = |V(G)|$ and size $m(G) = |E(G)|$. Two adjacent vertices in $G$ are \emph{neighbors}. The \emph{open neighborhood} of a vertex $v$ in $G$ is $N_G(v) = \{u \in V \, \colon \, uv \in E\}$ and the \emph{closed neighborhood of $v$} is $N_G[v] = \{v\} \cup N_G(v)$. We denote the degree of $v$ in $G$ by $\deg_G(v)$, and so $\deg_G(v) = |N_G(v)|$. The minimum and maximum degrees in $G$ are denoted by $\delta(G)$ and $\Delta(G)$, respectively. An \emph{isolated vertex} in $G$ is a vertex of degree~$0$ in $G$. A graph is \emph{isolate}-\emph{free} if it contains no isolated vertex. A vertex $v$ is a \emph{universal vertex}, also called a \emph{full vertex} in the literature, if $N_G[v] = V(G)$, that is, $\deg_G(v) = n(G) - 1$. If the graph $G$ is clear from the context, we simply write $V$, $E$, $n$, $m$, $\deg(v)$, $N(v)$, and $N[v]$ rather than $V(G)$, $E(G)$, $n(G)$, $m(G)$, $\deg_G(v)$, $N_G(v)$, and $N_G[v]$, respectively.

We denote a \emph{path} and \emph{cycle} on $n$ vertices by $P_n$ and $C_n$, respectively, and we denote a \emph{complete graph} on $n$ vertices by $K_n$. A \emph{complete bipartite graph} with partite sets of cardinalities $r$ and $s$ we denote by $K_{r,s}$. A \emph{star} is a complete bipartite graph $K_{1,s}$ where $s \ge 2$. A \emph{nontrivial tree} is a tree of order at least~$2$.  A partition of a set is a grouping of its elements into non-empty subsets, in such a way that every element of the set is included in exactly one subset.

A \emph{coalition} in a graph $G$ consists of two disjoint sets of vertices $X$ and $Y$ of $G$, neither of which is a dominating set but whose union $X \cup Y$ is a dominating set of $G$. Such sets $X$ and $Y$ \emph{form a coalition} in $G$. A \emph{coalition partition}, called a $c$-\emph{partition}, in $G$ is a partition $\Psi = \{V_1,\ldots,V_k\}$ of $V(G)$ such that for all $i \in [k]$, the set $V_i$ is either a singleton dominating set or forms a coalition with another set $V_j$ for some $j$, where $j \in [k] \setminus \{i\}$. The \emph{coalition number}, $C(G)$, in $G$ equals the maximum order $k$ of a $c$-partition of $G$. Coalitions in graphs were introduced and first studied by Haynes, Hedetniemi, Hedetniemi, McRae, and Mohan~\cite{coal0}, and have subsequently been studied, for example, in~\cite{coal1,coal2,coal3}. Their research primarily focused on examining coalition numbers in trees and cycles. In~\cite{coal2}, they established upper bounds on the coalition number of a graph in terms of its minimum and maximum degree.  Bakhshesh et al. in~\cite{bakhcoal}  characterized graphs $G$ of order $n$ with $\delta(G) \ge 1$ and $C(G) = n$. They also identified all trees $T$ of order~$n$ with $C(T) = n-1$.

In \cite{coal0}, Haynes et al. introduced the  \emph{refinement} of a coalition partition and defined a \emph{minimal coalition partition}. Let ${\cal A} = \{A_1,\ldots,A_r\}$ and ${\cal B}= \{B_1,\ldots, B_s\}$ be two partitions of $V(G)$. Partition ${\cal B}$ is a \emph{refinement} of partition ${\cal A}$, denoted ${\cal A}\le {\cal B}$, if every set $B_i \in {\cal B} $ is either equal to, or a proper subset of, some set $A_j \in {\cal A}$. Further if ${\cal A} \ne {\cal B}$, then ${\cal B}$ is a \emph{proper refinement} of ${\cal A}$, denoted ${\cal A}<{\cal B}$.
%
%
The following observation follows from the definition of a proper refinement of a coalition partition.

\begin{observation}{\rm (\cite{coal0})}
\label{obs0}
If $\Psi = \{V_1,\ldots,V_k\}$ is a $c$-partition of a graph $G$, and there exist two sets $V_i$ and $V_j$ whose union $V_i \cup V_j$ is not a dominating set, then the partition $\Psi'$ formed from $\Psi$ by replacing $V_i$ and $V_j$ with the union $V_i \cup V_j$ is a $c$-partition of $G$ and $\Psi$ is a proper refinement of $\Psi'$.
\end{observation}

A $c$-partition ${\cal A}$ is a \emph{minimal} $c$-\emph{partition} in $G$ if it is not a proper refinement of any other $c$-partition in $G$. In \cite{coal0}, Haynes et al. defined the \emph{minmin coalition number} $c_{\min}(G)$ of $G$ to equal the minimum order of a minimal $c$-partition of $G$.  A minimal $c$-partition of $G$ of cardinality~$c_{\min}(G)$ is called a $c_{\min}$-\emph{partition of $G$}. Haynes et al.~\cite{coal0} posed the following open problem.

\begin{problem}{\rm (\cite{coal0})}
\label{prob1}
{\rm What can you say about $c_{\min}(G)$?}
\end{problem}

This paper addresses Problem \ref{prob1}.  We proceed as follows. In Section~\ref{S:bounds}, we present lower and upper bounds on the minmin coalition number and we prove that if $G$ is a graph of order~$n$, then $2 \le c_{\min}(G) \le n$, and these bounds are sharp. In Section~\ref{S:charn}, we characterize graphs $G$ of order $n$ satisfying $c_{\min}(G) = n$. In Section~\ref{S:le4}, we give a comprehensive description of graphs $G$ satisfying $c_{\min}(G)=k$ where $2 \le k \le 4$. Additionally, we present a polynomial-time algorithm to determine if $c_{\min}(G)=2$ for a given graph $G$. If $G$ is an isolate-free graph that does not contain a universal vertex and has minimum degree~$1$, then we show that $c_{\min}(G)=2$. A characterization of graphs $G$ with minimum degree~$2$ and $c_{\min}(G) = 4$ is provided.

\section{Bounds on the minmin coalition number}
\label{S:bounds}

In this section, we present lower and upper bounds on the minmin coalition number. We prove firstly that the minmin coalition number is bounded above by the cardinality of an arbitrary $c$-partition of $G$.

\begin{proposition}
\label{prop:upperb}
If $G$ is a graph and ${\cal X}$ is a $c$-partition of $G$, then $c_{\min}(G) \le |{\cal X}|$.
\end{proposition}
\begin{proof}
Let ${\cal X}$ be an arbitrary $c$-partition of $G$. If ${\cal X}$ is a minimal $c$-partition of $G$, then $c_{\min}(G) \le |{\cal X}|$.  Now, assume that  ${\cal X}$ is not a minimal $c$-partition. By repeated applications of Observation~\ref{obs0}, there exists a minimal $c$-partition ${\cal P}$ of $G$ with ${\cal P} < {\cal X}$. Hence, $c_{\min}(G) \le |{\cal P}| < |{\cal X}|$.~\QED
\end{proof}

As an immediate consequence of Proposition~\ref{prop:upperb}, we have the following upper bound on the minmin coalition number.

\begin{corollary}
\label{cor:upperb}
If $G$ is a graph, then $c_{\min}(G) \le C(G)$.
\end{corollary}

We determine next a relationship between the minmin coalition number of a graph and the graph obtained from it by removing the universal vertices.

\begin{proposition}
\label{propGprim}
If $G$ is a graph that is not a complete graph and contains $k$ universal vertices, then $c_{\min}(G) = c_{\min}(G') + k$, where $G'$ is obtained from $G$ by removing the universal vertices.
\end{proposition}
\begin{proof}
Let $G$ be a graph that is not a complete graph and contains $k$ universal vertices, and let ${\cal X}$ be a $c_{\min}$-partition of $G$. Thus, ${\cal X}$ is a minimal $c$-partition of $G$ and $|{\cal X}| = c_{\min}(G)$. If $u$ is a universal vertex, then every $c$-partition of $G$ contains the set $\{u\}$. In particular, $\{u\} \in {\cal X}$. Let ${\cal X}'$ be obtained from ${\cal X}$ by removing all sets $\{u\}$, where $u$ is a universal vertex of $G$. The resulting partition ${\cal X}'$ is a $c$-partition of $G'$. Thus by Proposition~\ref{prop:upperb} we infer that $c_{\min}(G')\le |{\cal X}'| = |{\cal X}|-k = c_{\min}(G) - k$. To prove the reverse inequality, let ${\cal Y}'$ be a $c_{\min}$-partition of $G'$. Thus, ${\cal Y}'$ is a minimal $c$-partition of $G'$ and $|{\cal Y}'| = c_{\min}(G')$. Adding all sets $\{u\}$ to ${\cal Y}'$, where $u$ is a universal vertex of $G$, yields a partition ${\cal Y}$ that is a $c$-partition of $G$. By Proposition~\ref{prop:upperb} we infer that $c_{\min}(G) \le |{\cal Y}| = |{\cal Y}'|+k = c_{\min}(G') + k$. Consequently, $c_{\min}(G) = c_{\min}(G')+k$.~\QED
\end{proof}

\medskip
We next establish lower and upper bounds on the minmin coalition number of a graph.

\begin{theorem}
\label{thmf}
If $G$ is a graph of order $n \ge 2$, then $2 \le c_{\min}(G) \le n$, and these bounds are sharp.
\end{theorem}
\begin{proof}
Let $G$ be a graph of order $n \ge 2$. By Corollary~\ref{cor:upperb}, $c_{\min}(G) \le C(G)$. Since $C(G) \le n$, the upper bound $c_{\min}(G) \le n$ trivially holds. To prove the lower bound on $c_{\min}(G)$, let ${\cal Y}$ be a $c_{\min}$-partition of $G$. If ${\cal Y}$ contains a set of cardinality~$1$, then such a set is a singleton dominating set of $G$. In this case, we infer from the lower bound on the order of $G$ that $|{\cal Y}| \ge 2$. If every set in ${\cal Y}$ has cardinality at least~$2$, then every set in ${\cal Y}$ forms a coalition with some other set in the $c$-partition, implying once again that $|{\cal Y}| \ge 2$. Hence in both cases, $c_{\min}(G) = |{\cal Y}|  \ge 2$.

To prove the sharpness of lower bound, consider, for example, when $G$ is a path $v_1v_2 \ldots v_n$ of order~$n \ge 4$. The partition ${\cal X} = \{ \{v_1,v_2\}, \{v_3,\ldots,v_n\} \}$ is a $c$-partition of $G$ that is not a proper refinement of any other $c$-partition in $G$, implying that ${\cal X}$ is a minimal $c$-partition. Therefore, $2 \le c_{\min}(G) \le |{\cal X}| = 2$, and so $c_{\min}(G) = 2$. To prove the sharpness of upper  bound, consider, for example, a complete bipartite graph $G$ with partite sets $V_1, V_2, \ldots, V_k$ where $k \ge 2$ and $|V_i| = 2$ for all $i \in [k]$. Since $\gamma(G)=2$ and every subset of $V(G)$ of cardinality~$2$ is a dominating set of $G$, we infer that every $c$-partition of $G$ contains only singleton sets. Hence, $c_{\min}(G)=n$.~\QED
\end{proof}

\section{Graphs with large minmin coalition number}
\label{S:charn}

In this section, we characterize graphs $G$ of order $n$ satisfying $c_{\min}(G) = n$.  For this purpose, we define a family ${\cal M}$ of graphs $G$ that are generated in the following recursive manner. We begin by including the graphs $K_1$, $K_2$, and $\overline{K}_2$ in ${\cal M}$. If $H$ is a graph already present in ${\cal M}$, then we add the two graphs $K_1+H$ and $\overline{K}_2+H$ to ${\cal M}$, where the join $F + G$ of two graphs $F$ and $G$ is the graph formed from disjoint copies of $F$ and $H$ by joining every vertex of $F$ to every vertex of $G$. As an illustration, the $C_4 = \overline{K}_2 + \overline{K}_2$ belongs to the family~${\cal M}$, and so $C_4 \in {\cal M}$. The graph $\overline{K}_2 + C_4$ illustrated on the left hand drawing of Figure~\ref{fs4} (where here $H = C_4$) therefore belongs to the family~${\cal M}$. Moreover the graph $K_1 + C_4$ belongs to the family~${\cal M}$, implying that the graph $K_1 + H$ illustrated in the right hand drawing of Figure~\ref{fs4} (where here $H = K_1 + C_4$) therefore belongs to the family~${\cal M}$.

\begin{figure}[ht]
\begin{center}
\includegraphics[width=0.6\linewidth]{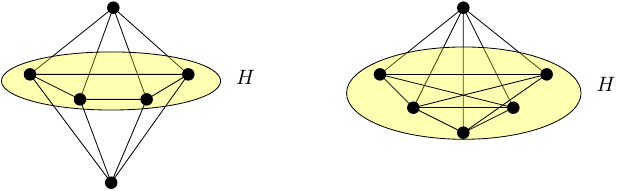}
\caption{Two graphs in the family ${\cal M}$.}
\label{fs4}
\end{center}
\end{figure}

\begin{proposition}
\label{prop:M}
If $G \in {\cal M}$ has order $n \ge 3$, then $G$ is a connected graph and $c_{\min}(G) = n$.
\end{proposition}
\begin{proof}
Let $G \in {\cal M}$ have order $n \ge 3$. By construction of graphs in the family ${\cal M}$, the graph $G$ is connected. Let $V=\{v_1,\ldots,v_n\}$. We proceed by induction on $n \ge 3$ to show that the property ($\star$) below holds in the graph $G$:
\begin{enumerate}
\item[($\star$)] Every two distinct vertices that are not universal vertices form a dominating set of $G$.
\end{enumerate}

Suppose that $n = 3$. In this case, either $G = K_1 + K_2 = K_3$ or $G = K_1 + \overline{K}_2 = P_3$. If $G = K_3$, then every vertex of $G$ is a universal vertex. If $G = P_3$, then $G$ has two vertices that are not universal and these two vertices form a dominating set of $G$. Hence if $n=3$, then property ($\star$) holds. Suppose that $n = 4$. Thus, $G = K_1 + H$ where $H \in \{P_3,K_3\}$ or $G = \overline{K}_2 + H$ where $H \in \{K_2,\overline{K}_2\}$. Hence, $G \in \{K_4,K_4 - e,C_4\}$ and property ($\star$) holds in the graph $G$. This establishes the base case when $n = 3$ and $n = 4$. Suppose that $n \ge 5$ and that if $G' \in {\cal M}$ has order~$n'$ where $3 \le n' < n$, then property ($\star$) holds in the graph $G'$. Let $G \in {\cal M}$ have order $n$. Since $G \in {\cal M}$, either $G = K_1 + G'$ or $G = \overline{K}_2 + G'$ for some graph $G' \in {\cal M}$.

Suppose firstly that $G = K_1 + G'$ for some graph $G' \in {\cal M}$ of order~$n'$. Necessarily, $n' = n-1$. Let $v$ be the vertex added to $G'$ to construct $G$, and so $v$ is a universal vertex of $G$. Every universal vertex in $G'$ is also a universal vertex in $G$. Let $x$ and $y$ be two arbitrary vertices of $G$ that are not universal vertices of $G$. Both $x$ and $y$ belong to $G'$ and neither $x$ nor $y$ is a universal vertices of $G'$. Applying the inductive hypothesis to $G'$, the set $\{x,y\}$ forms a dominating set of $G'$, and therefore also of $G$ since both $x$ and $y$ are adjacent to the universal vertex $v$ of $G$. Hence in this case when $G = K_1 + G'$ for some graph $G' \in {\cal M}$, property ($\star$) holds in the graph $G$.

Suppose secondly that $G = \overline{K}_2 + G'$ for some graph $G' \in {\cal M}$ of order~$n'$. Necessarily, $n' = n-2$. Let $v_1$ and $v_2$ be the two vertices added to $G'$ to construct $G$, and so $v_i$ is adjacent to every vertex of $G$ except for the vertex $v_{3-i}$ for $i \in [2]$. Every universal vertex in $G'$ is also a universal vertex in $G$. Let $x$ and $y$ be two arbitrary vertices of $G$ that are not universal vertices of $G$. If both $x$ and $y$ belong to $G'$, then neither $x$ nor $y$ is a universal vertex of $G'$, and so by the inductive hypothesis, the set $\{x,y\}$ forms a dominating set of $G'$, and therefore also of $G$. If exactly one of $x$ and $y$ belong to $G'$, then renaming vertices if necessary we may assume that $x = v_1$ and $y \in V(G')$. In this case, since $x$ dominates the set $V(G')$ and the vertex $y$ is adjacent to $v_2$, the set $\{x,y\}$ once again forms a dominating set of $G$. Finally, if $\{x,y\} = \{v_1,v_2\}$, then $\{x,y\}$ again forms a dominating set of $G$. Hence in this case when $G = \overline{K}_2 + G'$ for some graph $G' \in {\cal M}$, property ($\star$) holds in the graph $G$.

Since property ($\star$) holds in the graph $G$, if ${\cal X}$ is an arbitrary $c$-partition of $G$, then every set in the partition ${\cal X}$ is a singleton set, and so $|{\cal X}| = n$, implying that $c_{\min}(G)=n$.~\QED
\end{proof}

\medskip
We are now in a position to characterize graphs $G$ of order $n$ satisfying $c_{\min}(G) = n$.

\begin{theorem}
\label{thm:char-cmin-n}
If $G$ is a connected graph of order $n \ge 3$, then $c_{\min}(G) = n$ if and only if $G \in {\cal M}$.
\end{theorem}
\begin{proof}
Let $G$ be a connected graph of order $n \ge 3$. If $G \in {\cal M}$, then by Proposition~\ref{prop:M}, $c_{\min}(G) = n$. Hence it suffices for us to prove that if $c_{\min}(G) = n$, then $G \in {\cal M}$. We prove by induction on $n \ge 3$ that $G \in {\cal M}$. If $n = 3$, then since $G$ is a connected graph either $G = P_3$ or $G = K_3$. In both cases, $G \in {\cal M}$. This establishes the base case. Suppose that $n \ge 4$ and that if $G'$ is a connected graph of order~$n'$ where $3 \le n' < n$ and $c_{\min}(G') = n'$, then $G' \in {\cal M}$.

Let $G$ be a connected graph of order~$n$ satisfying $c_{\min}(G) = n$, and let $V=\{v_1,\ldots,v_n\}$. Thus, $G$ contains a unique $c$-partition, namely the partition in which every set is a singleton set. In particular, $G$ contains a unique $c_{\min}$-partition, namely the partition $\Psi = \{ \{v_1\}, \{v_2\}, \ldots, \{v_n\} \}$. Since the $c_{\min}$-partition $\Psi$ is not a proper refinement of any other $c$-partition in $G$, by Proposition~\ref{obs0} if $v_i$ and $v_j$ are distinct vertices that are not universal vertices of $G$, then the set $\{v_i,v_j\}$ is a dominating set of $G$ for all $i,j \in [n]$. If $G = K_n$, then by repeated applications of the join operation $K_1 + H$ where $H \in {\cal M}$, the graph $G \in {\cal M}$ noting that $K_2 \in {\cal M}$. Hence, we may assume that $G \ne K_n$, for otherwise the desired result is immediate. In particular, $G$ contains at least two vertices that are not universal vertices.

Let $v$ be an arbitrary vertex of $G$ that is not a universal vertex, and consider the set $S_v = V \setminus N_G[v]$. Thus, the set $S_v$ consists of all vertices of $G$ different from $v$ that are not adjacent to $v$. By supposition, $|S_v| \ge 1$. Suppose that $|S_v| \ge 2$, and let $\{x,y\} \subseteq S_v$. Since neither $x$ nor $y$ is adjacent to~$v$, the vertices $x$ and $y$ are not universal vertices of $G$. By our earlier observations, the set $\{x,y\}$ is therefore a dominating set of $G$. However, the vertex $v$ is not dominated by the set $\{x,y\}$, a contradiction. Hence, $|S_v| = 1$.

Let $S_v = \{v'\}$. Let $u$ be an arbitrary neighbor of $v$, and so $u \in N_G(v)$. If $u$ is a universal vertex, then $u$ is adjacent to $v'$. If $u$ is not a universal vertex, then by our earlier observations, the set $\{u,v\}$ is therefore a dominating set of $G$. In particular, this implies that the vertex $u$ is adjacent to $v'$. Hence, the vertex $v'$ is adjacent to every vertex in $N_G(v)$. Thus, $G = \overline{K}_2 + G'$ where $G' = G - \{v,v'\}$. Let $G'$ have order~$n'$, and so $n = n' - 2$. If $G'$ is a complete graph $K_{n-2}$, then $G' \in {\cal M}$, implying that $G' \in {\cal M}$.

Hence we may assume that $G'$ contains at least two vertices that are not universal vertices in $G'$. Let $x$ and $y$ be two distinct vertices that are not universal vertices of $G'$. Since $x$ and $y$ are not universal vertices of $G$, by our earlier observations that $\{x,y\}$ is a dominating set of $G$. This in turn implies that $\{x,y\}$ is a dominating set of $G'$. This is true for every two distinct vertices that are not universal vertices of $G'$. Hence, $c_{\min}(G') = n'$. If $n' = 2$, then either $G' = K_2$, in which case $G = K_4 - e$, or $G' = \overline{K}_2$, in which case $G = C_4$. In both cases, $G \in {\cal M}$. Hence, we may assume that $n' \ge 3$, for otherwise the desired result follows. Thus by our earlier properties of the graph $G'$, we infer that $G'$ is a connected graph. As observed earlier, $c_{\min}(G') = n'$. Applying the inductive hypothesis to $G'$, we infer that $G' \in {\cal M}$. This in turn implies that $G = \overline{K}_2 + G' \in {\cal M}$.~\QED
\end{proof}

\section{Graphs with small minmin coalition number}
\label{S:le4}

In this section, we study graphs with small minmin coalition number. We characterize graphs $G$ with no universal vertex satisfying $c_{\min}(G) = 2$. We present necessary and sufficient condition for a graph $G$ to satisfy $c_{\min}(G) \ge 3$, and necessary and sufficient condition for a graph $G$ to satisfy $c_{\min}(G) \ge 4$.  We first prove a necessary condition for a graph $G$ to satisfy $c_{\min}(G) = 2$.

\begin{lemma}
\label{lem:deg1}
If $G$ is an isolate-free graph with $\delta(G) = 1$ that does not contain a universal vertex, then $c_{\min}(G)=2$.
\end{lemma}
\begin{proof}
Let $G$ be an isolate-free graph with $\delta(G) = 1$ that does not contain a universal vertex. Suppose, to the contrary, that $c_{\min}(G) \ne 2$. Hence by Theorem \ref{thmf}, $c_{\min}(G)\ge 3$. In particular, this implies that $G$ has order at least~$3$. Let $x$ be a vertex of degree~$1$ in $G$, and let $y$ be its only neighbor. We now consider the partition $\Psi = (V_1,V_2)$ of $V(G)$ into sets $V_1 = \{x,y\}$ and $V_2 = V(G) \setminus V_1$. Since the vertex $x$ is not dominated by $V_2$, the set $V_2$ is not a dominating set of $G$. By supposition the graph $G$ does not contain a universal vertex. In particular, the vertex $y$ is not a universal vertex of $G$, implying that the set $V_1$ is not a dominating set of $G$. Hence, the sets $V_1$ and $V_2$ form a coalition in $G$. Thus, $\Psi$ is a $c$-partition of $G$, implying by Proposition~\ref{prop:upperb} that $c_{\min}(G) \le |\Psi| = 2$, a contradiction.~\QED
\end{proof}

\medskip
As a consequence of Lemma~\ref{lem:deg1}, we can determine the minmin coalition number of a tree.

\begin{corollary}
\label{cor:tree}
If $T$ is a nontrivial tree, then $c_{\min}(T)=2$, unless $T$ is a star $T \cong K_{1,r}$ where $r \ge 2$.
\end{corollary}
\begin{proof}
Let $T$ be a nontrivial tree. Thus, $T$ has order~$n \ge 2$. If $n = 2$, then $c_{\min}(T)=2$. Hence we may assume that $n \ge 3$, for otherwise the desired result is immediate. If $T$ does not contain a universal vertex, then we immediately infer from Lemma~\ref{lem:deg1} that $c_{\min}(G)=2$. Hence we may further assume that $T$ contains a universal vertex. Thus, $T$ is a star $K_{1,n-1}$. By our earlier assumptions, $n \ge 3$. If $\Psi$ is an arbitrary $c$-partition of $T$, then the universal vertex, $v$ say, of $T$ form a singleton set $\{v\}$ in $\Psi$. Since $n \ge 3$, no leaf of $T$ is a dominating set of $T$. Since the set of leaves in $T$ forms a dominating set of $T$, the set of leaves cannot be a set in the $c$-partition $\Psi$, implying that $\Psi$ contains at least two sets different from the singleton set $\{v\}$. Therefore, $|\Psi| \ge 3$. Since this is true for every $c$-partition of $T$, we infer that $c_{\min}(T) \ge 3$. On the other hand, if $(V_1,V_2)$ is an arbitrary partition of the set of leaves of $T$, then the partition ${\cal X} = (V_1,V_2,V_3)$ where $V_3 = \{v\}$ is a $c$-partition of $T$, and so by Proposition~\ref{prop:upperb} $c_{\min}(G) \le |{\cal X}| = 3$. Consequently, $c_{\min}(T) = 3$.~\QED
\end{proof}

\medskip
We present next a necessary and sufficient condition for a graph $G$ to satisfy $c_{\min}(G) \ge 3$.

\begin{theorem}
\label{thm:cmin3}
If $G$ is a graph that does not contain a universal vertex, then $c_{\min}(G) \ge 3$ if and only if for every vertex $v \in V$, the set $N[v]$ is a dominating set of $G$.
\end{theorem}
\begin{proof}
Let $G$ be a graph that does not contain a universal vertex. Suppose firstly that $c_{\min}(G) \ge 3$. Let $v$ be an arbitrary vertex of $G$, and let $S = N_G[v]$. Since the vertex $v$ is not adjacent to any vertex in $V \setminus S$, the set $V \setminus S$ is not a dominating set of $G$. If $S$ is not a dominating set of $G$, then $\{S,V \setminus S\}$ is a coalition partition of $G$, and so by Proposition~\ref{prop:upperb} $c_{\min}(G) \le 2$, a contradiction. Hence, $S$ is a dominating set of $G$.

Conversely, suppose that the set $N[v]$ is a dominating set of $G$ for every vertex $v \in V$. We show that $c_{\min}(G) \ge 3$. Suppose, to the contrary, that $c_{\min}(G)<3$. By Theorem \ref{thmf}, this implies that $c_{\min}(G)=2$. Let ${\cal X} = \{A, B\}$ be a $c_{\min}$-partition of $G$. Since $G$ has no universal vertex, neither set $A$ nor $B$ is a dominating set of $G$. Let $u$ be a vertex of $G$ not dominated by the set $A$, and so $A \cap N[u] = \emptyset$. Since $B = V \setminus A$, we have $N[u] \subseteq B$. By supposition the set $N[u]$ is a dominating set of $G$. Since the property of being a dominating set is superhereditary, the set $B$ is therefore a dominating set of $G$, a contradiction. Hence, $c_{\min}(G) \ge 3$, as desired.~\QED
\end{proof}

\medskip
As an immediate consequence of Theorem~\ref{thmf} and Theorem~\ref{thm:cmin3}, we have the following characterization of graphs with no universal vertex satisfying $c_{\min}(G) = 2$.

\begin{corollary}
\label{cor:cmin2}
If $G$ is a graph that does not contain a universal vertex, then $c_{\min}(G) = 2$ if and only if there exists a vertex $v \in V$ such that $N[v]$ is not a dominating set of $G$.
\end{corollary}

As an application of Corollary~\ref{cor:cmin2}, consider the Heawood graph and the Petersen graph illustrated in Figure~\ref{figH}(a) and (b), respectively. Since neither graph has a vertex whose closed neighborhood is a dominating set, we infer from Corollary~\ref{cor:cmin2} that the minmin coalition number is equal to~$2$ for both the Heawood graph and the Petersen graph.

\begin{figure}[hbt]
\begin{center}
{\subfloat[]{\includegraphics[width = 0.22\textwidth]{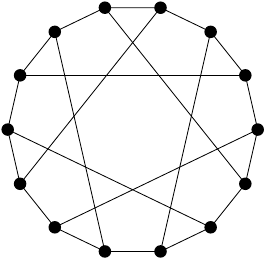}} }
\hspace{1cm}
{\subfloat[]{\includegraphics[width =0.22\textwidth]{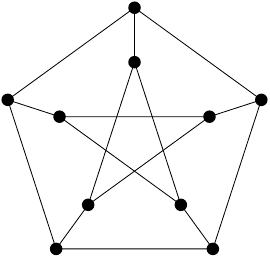}}}
\caption{(a): Heawood graph. (b): Petersen graph.}
\label{figH}
\end{center}
\end{figure}

As a further application of Corollary~\ref{cor:cmin2}, we can determine the minmin coalition number of a cycle.

\begin{corollary}
\label{cor:cycle}
The minmin coalition number of a cycle $C_n$ is given by following closed formula.
\[
c_{\min}(C_n) = \left\{
\begin{array}{cl}
2& \text{\rm if~} n\ge 6\\
3& \text{\rm if~} n=3 \text{~\rm or~} n=5.\\
4&\text{\rm if~} n=4.
\end{array}
\right.
\]
\end{corollary}
\begin{proof}
If $n \ge 6$, then by Corollary~\ref{cor:cmin2} we have $c_{\min}(C_n) = 2$. For $n \le 5$, Theorem~\ref{thm:cmin3} implies that $c_{\min}(C_n) \ge 3$. When $n=3$, the only $c$-partition of $C_3$ consists of three singleton sets which implies that $c_{\min}(C_3) = 3$. When $n=4$, we note that every subset of vertices of $C_4$ with at least two vertices is a dominating set, implying that the only $c$-partition of $C_4$ consists of four singleton sets, and so $c_{\min}(C_4) = 4$. Suppose that $n=5$ and consider the $5$-cycle $v_1 v_2 v_3 v_4 v_5 v_1$. The partition $\{\{c_1\}, \{c_2,c_3\},\{c_4,c_5\}\}$, for example, is a minimal $c$-partition of $C_5$, implying that $c_{\min}(C_5) = 3$.~\QED
\end{proof}

\medskip
We remark that Theorem~\ref{thm:cmin3} yields a polynomial time algorithm to determine whether a graph $G$ with no universal vertex satisfies $c_{\min}(G)=2$ or $c_{\min}(G)\ne 2$. We recursively examine all the vertices of $G$ one by one. If there exists a vertex $v \in V$ such that the set $N[v]$ is not a dominating set of $G$, then we infer by Theorem~\ref{thm:cmin3} that  $c_{\min}(G)=2$. If no such vertex exists, then we infer that $c_{\min}(G) \ne 2$. Given a set of vertices of $G$, one can easily check in polynomial time whether the set is a dominating set of $G$ or not. Therefore, the total time required to determine whether $c_{\min}(G) = 2$ is polynomial. This yields the following result.

\begin{theorem}
There exists a polynomial time algorithm to determine whether for a given graph $G$ with no universal vertices the equation $c_{\min}(G)=2$ holds or not.
\end{theorem}

\medskip
We present next a necessary and sufficient condition for a graph $G$ to satisfy $c_{\min}(G) \ge 4$.

\begin{theorem}
\label{thmc4}
If $G$ is a graph that does not contain a universal vertex, then $c_{\min}(G) \ge 4$ if and only if for every vertex $v \in V$ and every partition $\{P,Q\}$ of $N[v]$, at least one of $P$ or $Q$ is a dominating set of $G$.
\end{theorem}
\begin{proof}
Let $G$ be a graph that does not contain a universal vertex. Suppose firstly that $c_{\min}(G) \ge 4$. Let $v$ be an arbitrary vertex of $G$, and let $\{P,Q\}$ be an arbitrary partition of $N[v]$. Thus, $P \ne \emptyset$, $Q \ne \emptyset$ and $P \cup Q = N[v]$. Suppose that $P = \{v\}$, and so $Q=N(v)$. By Theorem~\ref{thm:cmin3}, $N[v]$ is a dominating set of $G$, implying that $Q$ is a dominating set, yielding the desired result. Analogously, if $Q = \{v\}$, then the desired result follows. Hence we may assume that $P \ne \{v\}$ and $Q\ne \{v\}$. Renaming the sets $P$ and $Q$ if necessary, we may assume that $v \in P$. Let $R = V \setminus (P \cup Q)$. Since $R$ contains no neighbors of $v$, it is not a dominating set of $G$. We now consider the partition ${\cal X} = \{P, Q, R\}$ of $V$. Suppose that neither $P$ nor $Q$ is a dominating set of $G$. Since $c_{\min}(G) \ne 2$, by Theorem~\ref{thm:cmin3} the set $P \cup Q$ is a dominating set. Additionally, $P \cup R$ is a dominating set of $G$ noting that the vertex $v$ dominates all vertices in $Q$. Therefore, $\cal X$ is a $c$-partition of $G$, which implies that $c_{\min}(G) \le 3$, a contradiction. Hence, at least one of $P$ or $Q$ is a dominating set.

Conversely, suppose that for every vertex $v \in V$ and every partition $\{P,Q\}$ of $N[v]$, at least one of $P$ or $Q$ is a dominating set of $G$. We show that $c_{\min}(G) \ge 4$. Suppose, to the contrary, that $c_{\min}(G) < 4$. Hence by Theorem \ref{thmf}, $c_{\min}(G)$ can only be equal to $2$ or $3$. If $c_{\min}(G)=2$, then by Corollary~\ref{cor:cmin2} there exists a vertex $v \in V$ such that $N[v]$ is not a dominating set of $G$. However in this case for every possible partition $\{P,Q\}$ of $N[v]$, neither $P$ nor $Q$ is a dominating set of $G$, a contradiction. Therefore, $c_{\min}(G) = 3$. Let ${\cal X} = \{A, B, C\}$ be a $c_{\min}$-partition of $G$. Thus, ${\cal X}$ is a minimal $c$-partition of $G$ and $|{\cal X}| = 3$. Since $G$ has no universal vertex, we note that none of the sets $A$, $B$ and $C$ is a dominating set of $G$. Let $u$ be a vertex of $G$ not dominated by the set $A$, and so $A \cap N[u] = \emptyset$. Since $B \cup C = V \setminus A$, we therefore have $N[u] \subseteq B \cup C$. Suppose that $B\cap N[u] = \emptyset$, implying that $N[u] \subseteq C$. By Theorem~\ref{thm:cmin3} and our assumption that $c_{\min}(G) \ne 2$, the set $N[u]$ is a dominating set. Since the property of being a dominating set is superhereditary, the set $C$ is therefore a dominating set of $G$, a contradiction. Hence, $B \cap N[u] \ne \emptyset$. Analogously, $C \cap N[u] \ne \emptyset$. Thus letting $P=B\cap N[u]$ and $Q=C\cap N[u]$, we have that $\{P,Q\}$ is a partition of $N[u]$. Hence by supposition, at least one of $P$ or $Q$ is a dominating set of $G$. This in turn implies that at least one of $B$ or $C$ is a dominating set of $G$, a contradiction. Hence, $c_{\min}(G) \ge 4$.~\QED
\end{proof}

\medskip
Let ${\cal F}$ be the family of graphs $G$ with vertex set $V = \{v,x,y\}\cup U$, where $\deg(v) = 2$ and $N(v)=\{x,y\}$, and where the vertices $x$ and $y$ are not adjacent but are both adjacent to all vertices of $U$ as illustrated in Figure~\ref{fs3}. Furthermore, the subgraph $G[U]$ induced by $U$ contains any number of edges, including the possibility of no edges. As an application of Theorem~\ref{thmc4}, we have the following result.

\begin{figure}[ht]
\begin{center}
	\includegraphics[width=0.25\linewidth]{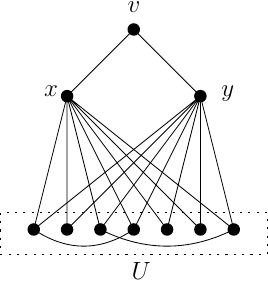}
	\caption{A graph in the family ${\cal F}$}
  \label{fs3}
	\end{center}
\end{figure}

\begin{theorem}
\label{thm:cmin4}
If $G$ is a graph with $\delta(G) = 2$ that does not contain a universal vertex, then $c_{\min}(G) = 4$ if and only if $G\in {\cal F}$.
\end{theorem}
\begin{proof}
Let $G$ be a graph with $\delta(G) = 2$ that does not contain a universal vertex. Let $v$ be a vertex of degree~$2$ in $G$, and let $N(v)=\{x,y\}$. Further, let $U = V \setminus N[v]$. Suppose firstly that $c_{\min}(G) = 4$. Let $P=\{v,x\}$ and $Q=\{y\}$. By Theorem \ref{thmc4}, at least one of $P$ and $Q$ is a dominating set of $G$. Since $G$ has no universal vertex, the set $Q$ is not a dominating set of $G$, implying that $P$ is necessarily a dominating set of $G$. This in turn implies that $x$ is adjacent to all vertices of $U$. Since $G$ has no universal vertex, the vertex $x$ is therefore not adjacent to the vertex~$y$. Interchanging the roles of $x$ and $y$, identical arguments show that the vertex $y$ is adjacent to all vertices of $U$. Hence, $G \in {\cal F}$.

Conversely, suppose that $G\in {\cal F}$. We adopt the notation used in the definition of the family ${\cal F}$, and so $V = \{v,x,y\}\cup U$, where $N_G(v) = \{x,y\}$. We show firstly that $c_{\min}(G) \ge 3$. Suppose, to the contrary, that $c_{\min}(G)=2$. Let ${\cal X} = \{A, B\}$ be a $c_{\min}$-partition of $G$. Since $G$ has no universal vertex, neither set $A$ nor $B$ is a dominating set of $G$. If $\{x,y\} \subseteq A$ or $\{x,y\}\subseteq B$, then since $\{x,y\}$ is a dominating set of $G$, at least one of $A$ or $B$ is a dominating set of $G$, a contradiction. Hence, $|A\cap\{x,y\}|=1$ and $|B\cap\{x,y\}|=1$. Renaming the sets $A$ and $B$ if necessary, we may assume that $x \in A$ and $y \in B$, and that $|A|\ge 2$. Let $w$ be a vertex in $A$ different from $x$. Thus either $w=v$ or $w\in U$. In both cases, the set $A$ is a dominating set of $G$, a contradiction. Hence, $c_{\min}(G) \ge 3$.

We show next that $c_{\min}(G) \ge 4$. Suppose, to the contrary, that $c_{\min}(G)=3$. Let ${\cal X} = \{A, B,C\}$ be a $c_{\min}$-partition of $G$. Since $G$ has no universal vertex, none of the sets $A$, $B$ or $C$ is a dominating set of $G$. We therefore infer that the set $\{x\}$ forms a singleton set in ${\cal X}$, as does the set $\{y\}$. Renaming the sets $A$, $B$ and $C$ we may assume that $A = \{x\}$ and $B = \{y\}$, and so $C = V \setminus \{x,y\}$. However, such a set $C$ is a dominating set of $G$, a contradiction. Hence, $c_{\min}(G) \ge 4$. The partition $\Psi = (\{v\},\{x\},\{y\},U)$ is a $c$-partition of $G$, implying by Proposition~\ref{prop:upperb} that $c_{\min}(G) \le 4$. Consequently,  $c_{\min}(G) = 4$.~\QED
\end{proof}

\section{Concluding remarks}

In this paper we address an open problem posed by Haynes, Hedetniemi, Hedetniemi, McRae, and Mohan~\cite{coal0} to study the minmin coalition number $c_{\min}(G)$ of a graph $G$. We show that if $G$ is a graph of order~$n$, then $2 \le c_{\min}(G)\le n$. Among other results, we characterized graphs $G$ of order $n$ satisfying $c_{\min}(G) = n$ and we provided polynomial time algorithm to determine if $c_{\min}(G) = 2$. It would be intriguing to investigate the open question of whether a polynomial time algorithm exists to determine larger values of $c_{\min}(G)$.

\end{document}